\documentclass{article}

\usepackage{amsmath,amsthm,amsfonts,amssymb,amscd,latexsym, fullpage,textcomp,amscd,diagrams}
\begin{document}
\date{ }
\theoremstyle{plain}
\newtheorem{thm}{\bf Theorem}[section]
\newtheorem{lem}[thm]{\bf Lemma}
\newtheorem{cor}[thm]{\bf Corollary}
\newtheorem{prop}[thm]{\bf Proposition}
\theoremstyle{remark}
\newtheorem{Case}{\bf Case}
\newtheorem{rem}[thm]{\bf Remark}
\newtheorem{claim}[thm]{\bf Claim}

\theoremstyle{definition}
\newtheorem{Def}{\bf Definition}
\newtheorem*{pf}{\bf Proof}
\newtheorem{Conj}{\bf Conjecture}

\newcommand{\nc}{\newcommand}
\newcommand{\rc}{\renewcommand}

\newcommand{\ca}{{\mathcal A}}
\newcommand{\BB}{{\mathcal B}} 
\newcommand{\CC}{{\mathcal C}}
\newcommand{\DD}{{\mathcal D}}
\newcommand{\EE}{{\mathcal E}}
\newcommand{\FF}{{\mathcal F}}
\newcommand{\GG}{{\mathcal G}}
\newcommand{\HH}{{\mathcal H}}
\newcommand{\II}{{\mathcal I}}
\newcommand{\JJ}{{\mathcal J}}
\newcommand{\KK}{{\mathcal K}}
\newcommand{\LL}{{\mathcal L}}
\newcommand{\MM}{{\mathcal M}}
\newcommand{\NN}{{\mathcal N}}
\newcommand{\OO}{{\mathcal O}}
\newcommand{\PP}{{\mathcal P}}
\newcommand{\QQ}{{\mathcal Q}}
\newcommand{\RR}{{\mathcal R}}
\newcommand{\TT}{{\mathcal T}}
\newcommand{\UU}{{\mathcal U}}
\newcommand{\VV}{{\mathcal V}}
\newcommand{\WW}{{\mathcal W}}
\newcommand{\ZZ}{{\mathcal Z}}
\newcommand{\XX}{{\mathcal X}}
\newcommand{\YY}{{\mathcal Y}}
\nc{\bba}{{\mathbb A}}
\nc{\bbb}{{\mathbb B}}
\nc{\bbc}{{\mathbb C}}
\nc{\bbd}{{\mathbb D}}
\nc{\bbe}{{\mathbb E}}
\nc{\bbf}{{\mathbb F}}
\nc{\bbg}{{\mathbb G}}
\nc{\bbh}{{\mathbb H}}
\nc{\bbi}{{\mathbb I}}
\nc{\bbj}{{\mathbb J}}
\nc{\bbk}{{\mathbb K}}
\nc{\bbl}{{\mathbb L}}
\nc{\bbm}{{\mathbb M}}
\nc{\bbo}{{\mathbb O}}
\nc{\bbp}{{\mathbb P}}
\nc{\bbq}{{\mathbb Q}}
\nc{\bbr}{{\mathbb R}}
\nc{\bbs}{{\mathbb S}}
\nc{\bb}{{\mathbb T}}
\nc{\bbu}{{\mathbb U}}
\nc{\bbv}{{\mathbb V}}
\nc{\bbw}{{\mathbb W}}
\nc{\bbx}{{\mathbb X}}
\nc{\bby}{{\mathbb Y}}
\nc{\bbz}{{\mathbb Z}}
\nc{\fA}{{\mathfrak A}}
\nc{\fB}{{\mathfrak B}}
\nc{\fC}{{\mathfrak C}}
\nc{\fD}{{\mathfrak D}}
\nc{\fE}{{\mathfrak E}}
\nc{\fF}{{\mathfrak F}}
\nc{\fG}{{\mathfrak G}}
\nc{\fH}{{\mathfrak H}}
\nc{\fI}{{\mathfrak I}}
\nc{\fJ}{{\mathfrak J}}
\nc{\fK}{{\mathfrak K}}
\nc{\fL}{{\mathfrak L}}
\nc{\fM}{{\mathfrak M}}
\nc{\fN}{{\mathfrak N}}
\nc{\fO}{{\mathfrak O}}
\nc{\fP}{{\mathfrak P}}
\nc{\fQ}{{\mathfrak Q}}
\nc{\fR}{{\mathfrak R}}
\nc{\fS}{{\mathfrak S}}
\nc{\fT}{{\mathfrak T}}
\nc{\fU}{{\mathfrak U}}
\nc{\fV}{{\mathfrak V}}
\nc{\fW}{{\mathfrak W}}
\nc{\fZ}{{\mathfrak Z}}
\nc{\fX}{{\mathfrak X}}
\nc{\fY}{{\mathfrak Y}}
\nc{\fa}{{\mathfrak a}}
\nc{\fb}{{\mathfrak b}}
\nc{\fc}{{\mathfrak c}}
\nc{\fd}{{\mathfrak d}}
\nc{\fe}{{\mathfrak e}}
\nc{\ff}{{\mathfrak f}}
\nc{\fh}{{\mathfrak h}}
\nc{\fj}{{\mathfrak j}}
\nc{\fk}{{\mathfrak k}}
\nc{\fl}{{\mathfrak{l}}}
\nc{\fm}{{\mathfrak m}}
\nc{\fn}{{\mathfrak n}}
\nc{\fo}{{\mathfrak o}}
\nc{\fp}{{\mathfrak p}}
\nc{\fq}{{\mathfrak q}}
\nc{\fr}{{\mathfrak r}}
\nc{\fs}{{\mathfrak s}}
\nc{\ft}{{\mathfrak t}}
\nc{\fu}{{\mathfrak u}}
\nc{\fv}{{\mathfrak v}}
\nc{\fw}{{\mathfrak w}}
\nc{\fz}{{\mathfrak z}}
\nc{\fx}{{\mathfrak x}}
\nc{\fy}{{\mathfrak y}}

\nc{\al}{{\alpha }}
\nc{\be}{{\beta }}
\nc{\ga}{{\gamma }}
\nc{\de}{{\delta }}
\nc{\ep}{{\varepsilon }}
\nc{\vap}{{\tepsilon }}

\nc{\ze}{{\zeta }}
\nc{\et}{{\eta }}
\nc{\vth}{{\vartheta }}

\nc{\io}{{\iota }}
\nc{\ka}{{\kappa }}
\nc{\la}{{\lambda }}
\nc{\vpi}{{     \varpi          }}
\nc{\vrho}{{    \varrho         }}
\nc{\si}{{      \sigma          }}
\nc{\ups}{{     \upsilon        }}
\nc{\vphi}{{    \varphi         }}
\nc{\om}{{      \omega          }}

\nc{\Ga}{{\Gamma }}
\nc{\De}{{\Delta }}
\nc{\nab}{{\nabla}}
\nc{\Th}{{\Theta }}
\nc{\La}{{\Lambda }}
\nc{\Si}{{\Sigma }}
\nc{\Ups}{{\Upsilon }}
\nc{\Om}{{\Omega }}

\nc{\zz}{{\mathbb Z}}
\newcommand{\N}{{\mathbb N}}
\newcommand{\etat}{\tilde{\eta}}
\newcommand{\sln}{\mathfrak{sl} _n }
\newcommand{\slnr}{{\mathfrak {sl} _n (\mathbb R)}}
\newcommand{\sun}{\mathfrak{su} _n}
\newcommand{\cc}{{\mathbb C}}
\newcommand{\rr}{{\mathbb R}}
\newcommand{\ac}{{\check{\alpha}}}
\newcommand{\orb}{{\mathcal O}}
\newcommand{\gcs}{{\mathcal{GC}n _{G/K}}}
\newcommand{\cds}{{\mathbb{C} \mathcal{D} _{G/K}}}
\newcommand{\ds}{{\mathcal{D} _{G/K}}}
\newcommand{\ggcs}{{\mathcal{GC} _{G/K} ^G}}
\newcommand{\gcds}{{\mathbb{C} \mathcal{D} _{G/K} ^G} }
\newcommand{\gds}{{\mathcal{D} _{G/K} ^G}}
\newcommand{\bc}{{\check{\beta}}}
\newcommand{\tep}{T_\varepsilon}
\newcommand{\vep}{\varepsilon}
\newcommand{\tepp}{T_{\overline{\varepsilon}}}
\newcommand{\ctsd}{\mathfrak{g} _\mathbb{C} ^*}
\newcommand{\gts}{$\mathfrak{g}\oplus\mathfrak{g}^{\ast}$}
\newcommand{\inv}{^{-1}}
\newcommand{\fg}{{\mathfrak g}}
\newcommand{\noi}{\noindent}
\newcommand{\lra}{{\longrightarrow}}
\newcommand{\pbd}{f^\star \mathcal{D}}
\newcommand{\cep}{{\overline{\varepsilon}}}
\newcommand{\cdt}{\tilde{\mathcal{D}}}
\newcommand{\cts}{{\mathfrak{g} _{\mathbb{C}}}}
\newcommand{\csts}{\mathfrak{k} _{\mathbb{C}}}
\newcommand{\cgts}{\mathfrak{g} _{\mathbb C} \oplus \mathfrak{g} _\mathbb{C} ^* }
\nc{\st}{{\; | \; }}
\nc{\trm}{\textreferencemark}
\nc{\tih}{{\tilde \HH}}
\nc{\vb}{\text{vector bundle}}
\nc{\vbm}{{\VV ect \BB und _M}}
\nc{\vs}{{Vect_{fd}}}
\nc{\cxn}{\text{connection}}
\nc{\bla}{{\mathfrak g ^* \rtimes \mathfrak g}}
\nc{\gc}{\text{generalized complex }}
\nc{\gcstr}{\text{generalized complex structure }}
\nc{\gcstrs}{\text{generalized complex structures }}
\nc{\spic}{\text{strict Picard $\infty$-category }}
\nc{\spics}{ { Pic ^\infty _{strict} }}
\nc{\imo}{{i-1}}
\nc{\imob}{ _\imo}
\nc{\imot}{ ^\imo }
\nc{\cab}{{C^{\leq 0}(\mathcal A b)}}
\nc{\qis}{\text{quasi-isomorphism }}
\nc{\ben}{	\begin{enumerate}\item		}
\nc{\een}{	\end{enumerate}			}
\nc{\bi}{    \begin{itemize}\item    }
\nc{\ei}{    \end{itemize}   }
\rc{\i}{  \item }

\newcommand{\lift}[2]{%
\setlength{\unitlength}{1pt}
\begin{picture}(0,0)(0,0)
\put(0,{#1}){\makebox(0,0)[b]{${#2}$}}
\end{picture}
}
\newcommand{\lowerarrow}[1]{%
\setlength{\unitlength}{0.03\DiagramCellWidth}
\begin{picture}(0,0)(0,0)
\qbezier(-28,-4)(0,-18)(28,-4)
\put(0,-14){\makebox(0,0)[t]{$\scriptstyle {#1}$}}
\put(28.6,-3.7){\vector(2,1){0}}
\end{picture}
}
\newcommand{\upperarrow}[1]{%
\setlength{\unitlength}{0.03\DiagramCellWidth}
\begin{picture}(0,0)(0,0)
\qbezier(-28,11)(0,25)(28,11)
\put(0,21){\makebox(0,0)[b]{$\scriptstyle {#1}$}}
\put(28.6,10.7){\vector(2,-1){0}}
\end{picture}
}

\newdiagramgrid{pentagon}%
   {0.618034,0.618034,1,1,1,1,0.618034,0.618034}%
   {1.17557,1.17557,1.902113,1.902113}

\nc{\dff}{{ \ \df\				}}
\begin{title}{The Quotient of a Category by the Action of a Monoidal Category}
\end{title}
\author{Brett Milburn \footnote{milburn@math.utexas.edu}}

\maketitle

\begin{abstract}
We introduce the notion of the quotient of a category $\CC$ by the action  $A :\MM \times \CC \lra \CC$ of a unital symmetric monoidal category $\MM$.  The quotient $\CC / \MM$ is a 2-category.  We prove its existence and uniqueness by first showing that every small 2-category has a presentation in terms of generators and relations and then describing the generators and relations needed for the quotient $\CC /\MM$.
\end{abstract}

\section{Introduction}
 We show that for any generating set $X$, there is a free 2-category $\ca _X$ on $X$.  Furthermore, given a generating set $X$ with relations $C$, there is a 2-category $\ca _{X,C}$ satisfying a universal property.  Moreover, any small 2-category has a presentation in terms of generators and relations.  We start by defining the weaker notion of a pre-2-category and showing the existence of free pre-2-categories and presentations of pre-2-categories by generators and relations.  We then apply the technology of pre-2-categories via generators and relations to attain the same results for 2-categories.  There are various versions of free n-categories in the literature \cite{gur}, \cite{scp}, \cite{str}, which are suitable in the appropriate contexts.  Schommer-Pries, for instance, considers free symmetric monoidal bicategories. Our interest in presenting 2-categories in terms of generators and relations is due to its utility in taking quotient categories.\\

Given a unital, symmetric monoidal category $\MM$ and an action $A :\MM \times \CC \lra \CC$ of $\MM$ on $\CC$, we would like to explain what it means to take the quotient $\CC / \MM$.  Our definition of the quotient is motivated by a more familiar quotient construction.  Given the action of a monoid $M$ on a space $X$, the smart notion of quotient $X/M$ is not a space but a category.  The objects of $X/M$ are the points of $X$, and morphisms in $X/M$ are indexed by $X \times M$. Instead of identifying points $x$ and $y = m.x$ in $X$ which are related by $m\in M$, there is a morphism $\zeta ^m _x$ from $x$ to $m.x$, thus remembering how $x$ and $y$ are related.  If $\MM$ is a symmetric monoidal category acting on a category $\CC$, we apply the same philosophy.  This time, however, the quotient $Q = \CC /\MM$ is a 2-category.  In addition to the 1-morphisms in $\CC$, objects of $\MM$ provide 1-morphisms $\zeta ^m _x : x \lra m.x$ for $x \in Ob(\CC)$, $m \in Ob(\MM)$.  We require that $\zeta$ is consistent with morphisms in $\CC$ and $\MM$ in a sense described by conditions Q4-Q6 in \S \ref{quotientsection}.  Roughly, consistency of $\zeta$ with morphisms in $\MM$ and $\CC$ means that we require certain diagrams to commute--ones that we would expect to commute in any reasonable definition of quotient.  However, instead of asking these diagrams to commute on the nose, we only require them to commute up to some 2-morphisms.  In section \ref{quotientsection} we define the quotient $\CC / \MM$ and demonstrate its existence and uniqueness up to isomorphism.

\section{2-Categories via Generators and Relations}

\noi We consider in the sequel only small n-categories and will only be concerned with $n$-categories for $n \leq 2$.  In Definition~\ref{def1} we recall the definition of 2-category but also define a weaker notion of pre-2-category, which is like a 2-category in that it has 0-objects, 1-morphism, 2-morphism and compositions but with none of the none of the associativity or coherence properties required of 2-categories.  \\

It is worth noting that we diverge from the standard nomenclature; what we mean by 2-category is what is often called a bicategory.  Additionally, we require morphisms between 2-categories to respect composition on the nose rather than up to 2-morphism.  Definition \ref{def1} follows the point of view of Street \cite{str}.  Instead of viewing an $n$-category as having 0-morphisms (i.e. objects), 1-morphisms, etc. as distinct, any $k$-morphism $x$, is identified with the $(k+1)$-morphism $id_x$.  In this way, all $k$-morphisms are on the same footing as members of the same set.  

\begin{Def}\label{def1}
\begin{enumerate}
\item Suppose that $0\leq n \leq \infty$. The data for a \emph{(small) strict $n$-category} is a set $\ca$ with maps $s_i , t_i : \ca \lra \ca $ for all $i < n$ and maps $*_i : \ca \times _\ca \ca \lra \ca$, where  $\ca \times _\ca \ca $ is the fibered product over maps $s_i : \ca \lra \ca $ and $t_i : \ca \lra \ca$.  Let $\rho _i ,\si _i \in \{s_i , t_i \}$ denote any source or target map. \\
\noi $(\ca , s_i , t_i , *_i )_{i < n}$ is said to be a strict $n$-category if the following 3 conditions are satisfied:
\begin{enumerate} 
\item For all $i <n$, $(\ca , s_i , t_i , *_i)$ is a category.  In other words, 
\begin{enumerate}
\item\label{1a} $\rho _i \sigma _i = \sigma _i$ for all $\si _i $, $\rho _i \in \{s_i , t_i\}$,
\item\label{1b} $a *_i s_i(a) = t_i (a) *_i a = a$,
\item\label{1c} $(a*_i b)*_i c = a *_i (b *_i c)$,
\item\label{1d} $s_i (a*_ib) = s_i b$, and $t_i (a*_i b) = t_i a$.
\end{enumerate}
\item For all $i < j$,$(\ca _i
, \ca _j)$ is a strict 2-category.  That is, for all $\si _i \in \{s_i , t_i \}$ and $\rho _j \in \{s_j, t_j\}$,
\begin{enumerate}
\item\label{2a}$\rho _j \sigma _i = \sigma _i$
\item\label{2b} $\sigma _i \rho _j = \sigma _i$
\item\label{2c} $\rho _j (a *_i b) = \rho _j a *_i \rho _j b$
\item\label{2d} $(a*_jb)*_i(\alpha *_j \beta) = (a*_i\alpha) *_j (b*_i \beta)$ when
one side is defined.  
\end{enumerate}
\end{enumerate}
Strict $n$-categories form a category $nCat _{str}$, the morphisms of which are maps of sets which respect all source, target, and composition maps. 

 
\item We define an \emph{n-categorically graded set} to be any set $S$ together with $s_i , t_i : S \lra S$, $0 \leq  i < n$ for some $0 \leq n \leq \infty$, satisfying 1.a.i, 1.b.i., and 1.b.ii above.  The collection of n-categorically graded sets are the objects of a category $gr_nCat$, the morphisms of which are the functions of sets which preserve the source and target maps in each degree.   
\item The category $p_nCat$ of \emph{pre-n-categories} has as objects $((\ca , s_i , t_i), *_i)$ for $1\leq i \leq n-1$, where $(\ca, s_i , t_i)$ is a an object of $gr_nCat$ together with compositions $*_i$ for $i \leq n-1$ satisfying 1.a.iv and 1.b.iii from above. Morphisms are maps of sets which preserve all structure maps.
\i  For a strict n-category, pre-n-category or n-categorically graded set $\ca$, $\ca _i := s_i \ca = t_i \ca$ is the set of of \emph{i-morphisms} or alternately \emph{i-objects}, which has the structure of an strict-i-category, pre-i-category, or i-categorically graded set, respectively. 
\item A \emph{2-category} is a pre-2-category $\ca$ with 2-isomorhphisms $\al _{h,g,f} :  (h*_0 g)*_0 f  \Longrightarrow  h *_0 (g*_0 f)$ for each $f,g,h \in \ca _1$, whenever the compositions are defined, as well as 2-isomorphisms $\lambda _f : t_0f *_0 f \Longrightarrow f$ and $\rho _f :f*_0 s_0f \Longrightarrow f$ for all 1-morphisms $f \in \ca _1$.  We require $\ca$ to satisfy the conditions described in \cite{gra}, \cite{lei}.  These conditions, which include strict associativity for $*_1$, are called \emph{coherence conditions for 2-categories}.  The collection of 2-categories are the objects of a category $2Cat$, the morphisms of which are morphisms of pre-2-categories which preserve the 2-morphisms $\al$, $\la$, $\rho$.
\end{enumerate}
\end{Def}

There are several useful functors relating the above categories, namely
\bi forgetful functors $2Cat \lra p_2Cat \lra gr_2Cat$ and more generally $p_nCat \lra gr_nCat$
\i full embeddings $gr_nCat \hookrightarrow gr_{(n+1)}Cat \hookrightarrow gr_\infty Cat$ and $p_nCat \hookrightarrow p_{(n+1)}Cat \hookrightarrow p_\infty Cat$ attained by letting $s_i = t_i = id$ for $i \geq n$.
\i a pair of forgetful functors $ p_{(n+1)}Cat \lra p_nCat$, the first of which is given by $\ca \mapsto \ca _n$ and the second forgets the higher structure maps.
\i Composing the previous forgetful functor n times, we get $Ob: p_nCat \lra p_0Cat \simeq Set$, which sends a pre-n-category to its underlying set.  Similarly, $Ob: gr_nCat \lra Set$ sends an n-categorically graded set to its underlying set.
\ei

\begin{Def} We call the morphisms in $2Cat$, $gr_nCat$, and $p_nCat$ \emph{maps} or \emph{functors}.  A map $F : \CC \lra \DD$ in $p_nCat$ or $gr_nCat$ is called \emph{injective} or \emph{surjective} if the underlying map $Ob(f)$ of sets is injective or surjective respectively.  More generally, set-theortic notions such as inclusions, intersections, etc. make sense in $gr_nCat$ and $p_nCat$ by considering the underlying sets.  We say, for instance, that $\CC \in p_nCat$ is a sub-pre-n-category of $\DD \in p_nCat$, written $\CC \subset \DD$, if $Ob(\CC) \subset Ob(\DD)$ and $\CC$ is closed under all $*_k$, $s_k$, $t_k$ in $\DD$.       
\end{Def}

\subsection{Pre-2-Categories}\label{pre2cats}

\noi {\bf Notation :} In $p_2Cat$, we let the symbol $*$ generically denote ``$*_0$ or $*_1$.''  In order to define free pre-2-categories, we will need to have formal strings or words representing composition. With this in mind, we denote a formal string of two objects in the following way.  For $Z \in gr_2Cat$ and subobjects $W,Y$ of $Z$, let $W\bullet Y = \{(w,i,y)\in W \times \mathbb Z /2\mathbb Z \times Y \st s_ix = t_iy \}$ and $W\bullet _j Y = \{(w,i,y)\in W \times \mathbb Z /2\mathbb Z \times Y \st s_ix = t_iy \; \textrm{and} \;  i = j \}$. For $w \in W$, $y \in Y$, let $w\bullet _i y = (w,i,y) \in W\bullet _i Y$, and let $w\bullet y $ generically denote ``$w\bullet _0y$ or $w \bullet _1y$.''

\begin{rem} Another way to view the notation $W \bullet Y$ is as follows.  We can view the correspondences $s_k , t_k : Y \rightrightarrows X$ over $X$ as a monoidal category with product $W \times ^k _X Y$ given by $W\bullet _kY$ defined above. If we were to construct free strict 2-categories from $Y$ over $X$, we would be interested in taking the free associative algebra $\sum _{n\geq 1} Y^{\times n}$, whereas in the construction of free pre-2-categories, we will be describing a refined version the free non-associative algebra $\sum _{n\geq 1} Y^{\times n} \times Tr_n$ of such a correspondence (where $Tr_n$ denotes all trees with n leaves).      
\end{rem}

We now show the existence of free pre-1-categories and pre-2-categories.  We will show the existence of a pre-2-category generated by a 2-categorically graded set $X$, but we would also like to consider the more general situation of generating a pre-2-category from a 1-categorically graded set $X_1$, which generates a free pre-1-category $\CC _{X_1}$ described in Lemma \ref{lem1} and two maps of sets $s_1, \, t_1: X_2 \rightrightarrows \CC _{X_1}$.

\begin{lem}\label{lem1} 
\begin{enumerate}
\item The forgetful functor $p_1Cat \lra gr_1Cat$ has a left adjoint $X \mapsto \CC _X$.  More explicitly, given $X \in gr_1Cat$, there exists an object $\CC _X \in p_1Cat$ with the property that there exists an inclusion $\iota _X : X \hookrightarrow \CC _X$ in $gr_1Cat$ and for any $D\in p_1Cat$ and $F \in Hom_{gr_1Cat}(X , D)$, $F$ factors uniquely through $\CC _X$, i.e. extends uniquely to a map $\tilde F \in Hom_{p_1Cat} (\CC _X ,D)$.  

\item Given the data of $(\CC , s_0 ,t_0 ,*_0) \in p_1Cat$ and a set $X_2$ together with maps of sets $s_1, t_1 :X_2 \rightrightarrows \CC$ such that $\sigma _0 s_1 = \sigma _0 t_1$ for all $\sigma _0 \in \{s_0,t_0\}$, 
\begin{enumerate}
\item The disjoint union $X = X_2 \cup \CC $ is a 2-categorically graded set.
\item  There exists $\FF _X \in p_2Cat$, called the free 2-pre-category on $X$, with the following property.  There is an inclusion $\iota _X : X \lra \FF _X$ in $gr_2Cat$, and if $D \in p_2cat$ and $F :X \lra D $ is a morphism in $gr_2Cat$ such that $F_{| \CC }$ is a map in $p_2Cat$, then $F$ extends uniquely to a map $\tilde F : \FF _X \lra D$ in $p_2Cat$.  
\end{enumerate}
\end{enumerate}
\end{lem}

\begin{proof} 
\begin{enumerate}
\item The pre-category $\CC = \CC _X$ is going to be built out of chains of length $n$ like the path category for $X$ except that $\CC$ keeps track of the order of composition, as we no longer require associativity.  We define chains of length $n$ recurssively by letting $S_1 = X$ and then defining $S_n = \bigsqcup _{1 \leq p \leq n-1} S_p \bullet _0 S_{n-p}$, where $S_p \bullet _0 S_q = \{ (x,y) \in S_p \times S_q \st s_0x = t_0 y \}$. We let $\CC = \bigsqcup _{1\leq n < \infty } S_n$.\\

\noi Define $s_0$, $t_0$ on $S_1$ to agree with the source and target maps already defined on $S_1 = X \in gr_1Cat$.  Now for $x\bullet _0 y \in S_p \bullet _0 S_q$, define $s_0( x\bullet _0 y) = s_0y$ and $t_0 (x\bullet _0 y) = t_0 y$.  Finally, composition on $\CC$ is defined as follows.  For $x \in S_p$, $y \in S_q$ such that $s_0 x = t_0 y$, $x *_0 y := x\bullet _0 y \in S_{p+q}$.  One may easily check that $\CC \in gr_1Cat$ and that conditions \ref{1d} and \ref{2c} of Definition~\ref{def1} so that $\CC \in p_1Cat$.  \\

\noi Given $D \in p_2Cat$ and $F : X \lra D$ in $gr_2Cat$, we would like to extend $F$ to a map $\tilde F : \CC \lra D$ of pre-2-categories.  We must have $\tilde F_{|S_1} = F$.  Now, having defined $\tilde F_{|S_k}$ for $k <n$, if $x\bullet _0 y \in S_p \bullet _0 S_{n-p} \subset S_n$, letting $\tilde F(x\bullet _0 y) = Fx*_0Fy$ defines $\tilde F$ on all of $\CC$, and obviously, in order to respect composition, this is the only possible choice for $\tilde F$.  

\item 
\begin{enumerate}
\item We define $s_i$, $t_i$ so that on $\CC $, $s_0$, $t_0$ agree with the source and target maps for $\CC \in gr_1Cat \subset gr_2Cat$ and $(s_1)_{|\CC} , (t_1) _{|\CC} = id_\CC$.  On $X_2$, we let $s_1$, $t_1:  X_2 \lra \CC $ be the maps specified above, and for $\sigma _0 \in \{s_0 , \, t_0 \}$, we let ${\sigma _0} _{|X_2} = \sigma _0 s _1 : X_2  \lra \CC$ or equivalently $\sigma_0 t_1$.  It is trivial to verify that properties \ref{1a}, \ref{2a}, and \ref{2b} of definition~\ref{def1} are satisfied.  

\item Let $S_1 = X_2 \bigsqcup \CC$, let $S_2 =( S_1 \bullet S_1)\setminus \CC \bullet _0 \CC$, and $S_n = \bigsqcup S_p \bullet S_{n-p}$ for $n > 2$.  Now we define $\FF _X = \bigcup _{1\leq n < \infty } S_n$.  Let $s_0 (x\bullet y) = s_0y$, $t_0 (x\bullet y) = t_0 x$, $s_1 (x \bullet _1 y ) = s_1y$, $t_1(x\bullet _1 y) = t_1 x$, and $\sigma _1 (x\bullet _0 y) = \sigma _1 x *_0 \sigma _1 y$. To see that this composition makes sense, an easy inductive proof shows that $s_1(S_p), \; t_1(S_p) \subset \CC$ for all $p$. Note that $(\FF _X)_1 = \CC$.  With these source and target maps, $\FF _X $ is a 2-categorically graded set.  There are composition laws on $\FF_X$ as follows. 
\begin{displaymath} 
x*_0 y = 
\left\{ 
\begin{array}{ll}
x *_0 y    & \textrm{ if $x,y \in \CC $}\\
x\bullet_0y    & \textrm{otherwise}
\end{array}
\right.
\end{displaymath}
\noi and $x*_1y = x\bullet _1 y$. One can check that $\FF _X \in p_2Cat$.\\   

Suppose $F : X \lra D$ is a map in $gr_2Cat$ such that $F$ restricted to $\CC$ is a map in $p_2Cat$.  Having defined $\tilde F$ on $S_k$ for $k\leq n$, define $\tilde F$ on $S_{n+1}$ by $\tilde F(x \bullet _i y) = \tilde Fx *_i \tilde Fy$ for $i \in \{0,1\}$. Clearly $\tilde F$ is well defined, and  $\tilde F(x*_iy) = \tilde F x *_i \tilde F y$.  Furthermore, as $X \subset \FF _X$ in $gr_2Cat$, it is apparent that $\tilde F$ is the only possible extension of $F$ to a map $\tilde F \in Hom_{p_2Cat}(\FF _X , D)$.  
\end{enumerate}
\end{enumerate}
\end{proof}

\begin{cor} The forgetful functor $p_2Cat \lra gr_2Cat$ is left adjoint to the functor which sends $X \in gr_2Cat$ to $\FF _{(X \setminus X_1) \cup \CC _{X_1}} \in p_2Cat$.
\end{cor}
\begin{proof} This is a special case of Lemma \ref{lem1}. Suppose $X \in gr_2Cat$.  Let $\CC = \CC _{X_1}$ and $X^\prime = (X \setminus X_1) \cup \CC$.  We take $X \setminus X_1$ instead of all of $X$ in order to avoid having redundant 1-morphisms.  By composing $s_1 , t_1 : (X \setminus X_1) \rightrightarrows X_1$ with the inclusion $X_1 \hookrightarrow \CC$ to get maps $(X\setminus X_1)\rightrightarrows \CC$, part 2a of Lemma~\ref{lem1} guarantees that $X^\prime$ is a 2-categorically graded set. \\

Given $D \in p_2Cat$ and a map $X \stackrel {F}\lra D$ in $gr_2Cat$, we aim to give a map $\FF _{X^\prime} \lra D$ in $p_2Cat$ and show that this assignment $Hom_{gr_2Cat} (X,D) \lra Hom_{p_2Cat}(\FF_{X^\prime} , D)$ is an isomorphism. By Lemma~\ref{lem1} part 1, $F_{|X_1} \in Hom_{gr_1Cat}(X_1 , D) \simeq Hom_{p_1Cat}(\CC _{X _1} , D)$.  Here we consider $D$ as a pre-1-category by forgetting the higher structure maps.  Note also that $\CC \in p_1Cat \hookrightarrow p_2Cat$ and $Hom_{p_1Cat}(\CC , D) \simeq Hom_{p_2Cat}(\CC , D)$.  Since we have extended $F $ from $X_1$ to $\CC$, this allows us to extend $F$ uniquely from $X\subset X ^\prime$ to a map $F : X^\prime \lra D$ in $gr_2Cat$ such that $ F_{|\CC} : \CC \lra D$ is a map of pre-2-categories.  By Lemma~\ref{lem1}(2b), $F : X \lra D$ extends uniquely to a map $\tilde F : \FF _{X ^\prime } \lra D$ in $p_2Cat$.  By the uniqueness of the extensions, the map $Hom_{gr_2Cat}(X,D)\lra Hom_{p_2Cat}(\FF_{(X \setminus X_1)\cup \CC _{X_1}} , D)$ is an inclusion.  Since $X \subset X ^\prime \subset \FF _{X^\prime} $ in $gr_2Cat$, every map $G : \FF _{X ^\prime } \lra D$ in $p_2Cat$ is an extension of $G_{|X} : X \lra G$ in $gr_2Cat$, whence $Hom_{gr_2Cat}(X,D)\simeq Hom_{p_2Cat}(\FF_{(X \setminus X_1)\cup \CC _{X_1}} , D)$.  
\end{proof}

\begin{Def} 
\begin{enumerate}
\item  As in Lemma \ref{lem1}, given the data  $X = (X_1 , X_2 \rightrightarrows \CC _{X_1})$ of $X_1 \in gr_1Cat$ (which defines $(\CC _{X_1} , s_0 ,t_0, *_0) \in p_1Cat$) and a set $X_2$ with maps of sets $s_1, t_1 :X_2 \lra \CC _{X_1}$ such that $\sigma _0 s_1 = \sigma _0 t_1$ for all $\sigma _0 \in \{s_0,t_0\}$, the \emph{pre-2-category generated by $X$} is the free pre-2-category $\FF _{X_2 \cup \CC _{X_1} }$, which by abuse of notation we also denote by $\FF _X$. The data $X$ is the \emph{generating data} for the pre-2-category $\FF _X$. We also write $X = X_1 \cup X_2$ for brevity.

\item A set of \emph{conditions} on generating data $X$ is a binary relation on $\FF _X$. 
\end{enumerate} 
\end{Def}

\begin{lem}\label{firstquotientlem} Given generating data $X$ and conditions $C$, there exists an equivalence relation $\sim$ on $\FF _X$ such that $\FF _X / \sim \,  \in gr_2Cat$ and has the property that for any $D \in p_2Cat$ and $F \in Hom_{p_2Cat} (\FF_X , D) $ such that $xCy $ implies $F(x) = F(y)$ for $x,y \in \FF _X$, $F$ factors through $\FF _X \lra \FF _X / \sim$ in $gr_2Cat$.  
\end{lem}

\begin{proof} Let $\sim$ denote the finest relation on $ \FF _X$ satisfying the following conditions:

\begin{displaymath}
\begin{array}{lll}
\textrm{P0:  $\sim$ is an equivalence relation.} \\

\textrm{P1: If $xCy$, then $x\sim y$.}\\

\textrm{P2: If $x\sim y$, then $\sigma _i x \sim \sigma _iy$ for $\sigma _i \in \{s_0,t_0,s_1,t_1\}$.}\\

\textrm{P3: If $x\sim x^\prime$ and $y \sim y^\prime$, then $x\bullet y \sim x^\prime \bullet y^\prime$ whenever both compositions are defined.}
\end{array}
\end{displaymath}

The notation in P3 is explained at the beginning of \S \ref{pre2cats} and in the proof of \ref{lem1}(2b).  Letting $x \sim y$ for all $x,y \in \FF _X$ is such a relation.  Because P0-P3 are closed under interesctions (i.e. mutual refinements), Zorn's lemma ensures the existence of a finest relation satisfying P0-P3. \\

\noi Now suppose $F : \FF \lra D$ as above.  Then the relation $xRy$ if $Fx = Fy$ satisfies P0-P3.  Thus, $F$ factors through $\FF_X / R \in gr_2Cat$.  Since $\sim$ is the smallest such relation, $\FF \lra \FF / R$ factors through $\FF /\sim$.  Hence, $F$ also factors through $\FF / \sim$.   
\end{proof}

We now show that for any generating set $X$ and conditions $C$, there is a pre-2-category $\FF _{X/C}$ generated by $X$ and satisfying $C$.

\begin{thm}\label{thm1} Given generating data $X = X_1 \cup X_2$ and conditions $C$, there exists a unique $\FF_{X /C} \in p_2Cat$ satisfying:
\ben There is a map $G : \FF _X \lra \FF_{X/C} $ in $p_2Cat$ such that for all $x,y\in \FF _X$, $xCy$ implies $G(x) = G (y)$.
\i $\FF _{X/C} $ is universal among pre-2-categories satisfying the above property in the sense that for any other
 map $F : \FF _X \lra D$ in $p_2Cat$ for which $xCy$ implies $Fx = Fy$ for all $x,y\in \FF _X$, $F$ factors uniquely through $G$ as seen in the diagram in $p_2Cat$
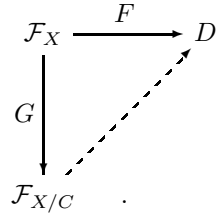
\begin{diagram}
\FF _X     & \rTo^F     & D\\
\dTo^G     & \ruDashto\\
\FF _{X/C} &.
\end{diagram}
\een
\end{thm}

\begin{proof}
First we consider only 0-objects and 1-morphisms to get a quotient category $\CC ^\prime $ from $\CC = \CC _{X_1}$.  The relation $\sim$ on $\FF _X$ of Lemma \ref{firstquotientlem}  restricts to an equivalence relation on $\CC = (\FF _X)_1$.  That is to say, for $x,y \in \CC$, $x \sim y$ in $\CC$ if and only if $x \sim y $ in $\FF _X$.  Additionally, $\overline \CC := \CC /\sim \in gr_2Cat$ because $\sim$ satisfies P2.  Now we define $\CC ^\prime$ by taking $S_1 = \overline \CC$, $S_2 = \{ \overline x \bullet \overline y \st \overline x ,y \in S_1 \;\textrm{and for all}\;  x^\prime \sim x \; , \; y^\prime \sim y, \; x^\prime *_0 y^\prime \; \textrm{is not defined} \}$.  We define $S_n = \bigsqcup _{0<p<n} S_p \bullet S_{n-p} $ for all $n>2$ and $\CC ^\prime = \bigcup _{n= 1}^\infty S_n$.  Define $s_0(x\bullet y ):= s_0y$, $t_0 (x\bullet y ) := t_0 x$. Composition is defined as 

\begin{displaymath} 
\overline x*_0 \overline y = 
\left\{ 
\begin{array}{ll}
\overline {x^\prime  *_0 y^\prime}    & \textrm{ if $x^\prime *_0 y^\prime \in \CC$ is defined for some $\CC \ni x ^\prime \equiv x$, $\CC \ni y ^\prime \equiv y$}\\
\overline x\bullet _0 \overline y =   & \textrm{otherwise}
\end{array}
\right.
\end{displaymath}

\noi so that $\CC ^\prime \in gr_1Cat$.  This composition gives $\CC ^\prime $ the structure of a pre-1-category.  \\

\noi We claim that any map $F: \CC \lra D$ of pre-1-categories such that $xCy$ implies $Fx=Fy$ must factor through $\CC ^\prime$. Such a map $F: \CC \lra D$ must factor through  $\overline \CC \in gr_1Cat$, which can be extended to a map $\tilde F : \CC ^\prime \lra D$ in $p_1Cat$ via $\tilde F (\overline x \bullet _0 \overline y) = F(x) *_0 F ( y)$.\\


 We now have $X_2 \rightrightarrows \CC \lra \CC ^\prime$, making $X^\prime = X_2   \cup \CC^\prime $ a categorically graded set with a map $\CC \cup X_2 \lra \CC ^\prime \cup X_2$ in $gr_2Cat$.  This induces $H: \FF _X \lra \FF _{X ^\prime}$ in $p_2Cat$.  The next step is to identify all remaining 2-morphisms related by $C$.  We therefore want a relation $\sim$ on $\FF _{X ^\prime}$ which is the finest relation satisfying:
\begin{displaymath}
\begin{array}{lll}
\textrm{P0:  $\sim$ is an equivalence relation.} \\

\textrm{P1$^\prime$: If $xCy$ for $x,y \in \FF _X$, then $Hx\sim Hy$.}\\

\textrm{P2: If $x\sim y$, then $\sigma _i x \sim \sigma _iy$ for $\sigma _i \in \{s_0,t_0,s_1,t_1\}$.}\\

\textrm{P3: If $x\sim x^\prime$ and $y \sim y^\prime$, then $x\bullet y \sim x^\prime \bullet y^\prime$ whenever both compositions are defined.}\\
\textrm{P4: If $x,y \in \CC ^\prime \subset \FF _{X^\prime}$, then $x\sim y $ implies $x =y$.}
\end{array}
\end{displaymath}
\noi Suppose there exists such a relation.  Conditions P0-P4 are closed under taking refinements of two such relations.  Zorn's lemma implies that there is a minimal such relation $R$.  Let $\FF _{X / C} := \FF _{X^\prime}/R$.  Properties P2 and P3 guarantee that $\FF _{X/C}$ is a pre-2-category.  We wish to show that $\FF _{X /C}$ has the specified universal property.  To this end, let $F : \FF _X \lra D$ be any map in $p_2Cat$ such that $Fx  = Fy$ whenever $xCy$.  Then $F_{|\CC} : \CC \lra D$ factors uniquely through $\CC ^\prime$ as we have already shown, thus inducing a unique map $F ^\prime :\FF _{X ^\prime} \lra D$ in $p_2Cat$.  Define a relation $Q$ on $\FF _{X^\prime} $ by $xQy$ if $x$ and $y$ lie in the same fiber of $F ^\prime$.  Conditions P0-P3 above are satisfied by $Q$.  Clearly, since $R$ is the finest relations satisfying P0-P4, it is also the finest relations satisfying P0-P3.  Hence, $F^\prime$ factors uniquely through $\FF _{X ^\prime}/Q$, which factors uniquely through $\FF _{X ^\prime} / R$ in $p_2Cat$ via the map $\pi : \FF _{X ^\prime}/R \lra \FF _{X ^\prime}/Q$.  Therefore, $F$ factors uniquely through $\FF _X \lra \FF _{X^\prime} /R$ as desired.  This can be expressed in the following commutative diagram in $p_2Cat$
\begin{diagram}
\FF _{X^\prime}             & \rTo^{F^\prime}     & D\\
\dTo               & \rdTo    & \uTo\\
\FF _{X^\prime}/R  & \rTo_\pi     &\FF _{X^\prime}/Q .
\end{diagram}

\noi It only remains to show that there exists a relation on $\FF _{X ^\prime}$ satisfying P0-P4. In general, let $\ca \in p2cat$ and $\CC = \ca _1$.  Given a subset $\mathcal S \subset \CC \times \CC$ such that:
\bi $\CC \simeq \Delta \CC \subset \mathcal S$,
\i $\sigma _0 \pi _1 = \sigma _0 \pi _2$ on $\mathcal S$ for all $\si _0 \in \{s_0, t_0 \}$,
\i if $(f,g)$,$(h,k) \in \mathcal S$ satisfy $t_0 h =s_0 f$, then $(fh,gk)\in \mathcal S$, and
\i  $(h,g), \; (g,f)\in \mathcal S$ implies $(h,f) \in \mathcal S$,
\ei
 then $\mathcal S$ is a pre-2-category with stucture maps $s _1  = \Delta \pi _1$, $t_1 = \Delta \pi _2$, $s_0 = \Delta s_0 \pi _1$, $t_0 = \Delta t_0 \pi _2$, $(f,h)*_0(g,k) = (f*_0g , h*_0k)$, and $(f,h)*_1(h,k) = (f,k)$. The important point is that $\mathcal S$ has the property that for every $f,g \in \mathcal S _1 \simeq \CC $, there exists at most one 2-morphism from $f$ to $g$.  Now, starting from $\FF _{X ^\prime}$, let 
$\mathcal S = \{ (f,g) \in \CC ^\prime \st \textrm{there exists a 2-morphism} \; \al : f \Longrightarrow g \}$.  Then there is a projection $\pi : \FF _{X ^\prime} \lra  \mathcal S$, and the fibers of $\pi$ determine a relation satisfying P0-P4. 
\end{proof}

\subsection{2-Categories}

In order to apply the previous results to 2-categories, we observe that a 2-category is simply a pre-2-category with extra data and conditions. 

\begin{thm}\label{thm2}Let $X = X_1 \cup X_2$ be generating data and impose conditions $C$.  There exists a unique (up to isomorphism) 2-category $\ca _{X,C}$ equipped with a map $G : \FF _X \lra \ca _{X,C}$ in $p_2Cat$ such that $G(x) = G(y)$ whenever $xCy$ and such that $\ca _{X,C}$ is universal with respect to this property in the following sense.  Given a 2-category $D$ and a map $F : \FF _X \lra D$ in $p_2Cat$ such that for all $x,y\in \FF _X$,  $xCy$ implies $F(x)= F(y)$, $F$ factors uniquely through $G$ in $p_2Cat$ in such a way that the map $H:\ca _{X,C} \lra D$ such that $HG = F$ is a map of 2-categories.  We call $\ca _{X,C} $ the 2-category generated by $X$ with conditions $C$. 
\end{thm}

\begin{proof} This is only a slight modification of the proof of Thorem \ref{thm1} where the generating data is enlarged to contain the structure morphisms $\al _{f,g,h},\la _f , \rho _f$ and we add to $C$ coherence conditions for 2-categories.  We work under the assumption that the generating data $X$ does not already contain the structure 2-morphisms for a 2-category.\\

Beginning with $\FF _{X^\prime}$, the pre-2-category defined in the third paragraph of the proof of Theorem \ref{thm1}, we add to $X^\prime$ 2-morphisms $\lambda _f : {t_0 f} *_0 f \Longrightarrow f$ and  $\rho _f : f *_0 {s_0 f} \Longrightarrow f$ for each $f \in \CC ^\prime = (\FF _{X^\prime} )_1$ as well as a 2-morphism  $\al _{h,g,f} : h*_0 (g*_0f) \Longrightarrow (h*_0 g )*_0 f$ for each triple of $f,g,h$ of 1-morphisms in $\CC ^\prime$. Also we add 2-morphisms $\al _{h,g,f} \inv :(h*_0 g )*_0 f   \Longrightarrow h*_0 (g*_0f)$, $\rho _f \inv : f \Longrightarrow f *_0 {s_0 f} $, and $\la _f \inv : f \Longrightarrow  {t_0 f} *_0 f$ which are going to be the inverses of $\al _{h,g,f}$, $\la _f$, $\rho _f$ respectively in the 2-category $\ca _{X,C}$.  Let $Y = X^\prime \bigsqcup \{ \al _{f,h,g} , \lambda _f , \rho _f , \al _{f,g,h} \inv , \rho _f \inv , \la _f \inv \}_{f,g,h \in \CC ^\prime} $ and $C^\prime = C \cup \{\textrm{coherence conditions for a 2-category}\}\cup I$.  Here $I$ denotes the set of relations $ \{ (\al _{h,g,f} \bullet_1 \al _{h,g,f} \inv , h\bullet_0 (g\bullet_0 f)), (\al _{h,g,f} \inv \bullet_1 \al _{h,g,f}  , (h\bullet_0 g)\bullet_0 f), (\la _f \bullet_1 \la _f \inv , f) , (\la _f \inv \bullet _1 \la _f , t_0 f \bullet _0 f), (\rho _f \bullet _1 \rho _f \inv , f), (\rho _f \inv \bullet _1 \rho _f , f\bullet _0 s_0f)\}$, where we think of the binary relation $C ^\prime $ as a subset of $Ob(\FF _{Y}) \times Ob( \FF _{Y})$. \\

 The inclusion (i.e. injective map) $X ^\prime \hookrightarrow Y$ in $gr_2Cat$ induces an inclusion $\FF _{X^\prime} \hookrightarrow \FF _Y$ in $p_2Cat$. Let $R$ be the relation on $\FF _{X ^\prime}$ (described in the penultimate paragraph of the proof of Theorem \ref{thm1}) such that $\FF _{X/C} = \FF _{X ^\prime } /R $.  Now we let $R ^\prime$ be the finest binary relation on $\FF _Y$ satisfying P0, P2-P4 and having $C^\prime$ and $R$ as refinements.  The existence of a minimal relation $R^\prime$ is proven by the same arguments used in the proof of Theorem \ref{thm1}.  The quotient $\ca _{X,C} := \FF _Y /R^\prime$ is a pre-2-category containing $\FF _{X/C}$ as a subcategory (in the sense that there is an inclusion $\FF _{X/C} \hookrightarrow \FF _Y / R ^\prime$).  The generating data $Y$ contains the extra data needed to make $\FF _{X ^\prime}$ into a 2-category, and the conditions $C^\prime$ are chosen  for the purpose of ensuring that $\FF _Y / R ^\prime$ satisfies the coherence conditions for 2-categories.  \\

More precisely, in order for $\ca _{X,C} $ to be a 2-category, it must contain 2-isomorphism $\al _{f,g,h}$ $\la _f$, and $\rho _f$ for all $ f,g,h \in (\ca _{X,C})_1$, and $\ca _{X,C} $ must satisfy the coherence conditions.  One obstacle to $\ca _{X,C} $ to be a 2-category is that we have not added enough $\al$'s $\rho$'s and $\la $'s.  We have added an $\al _{f,g,h}$ $\rho _f$ and $\la _f$ for all $f,g,h \in \CC ^\prime \subset (\ca _{X,C})_1$, but we need one for each $f,g,h \in (\ca _{X,C})_1$.  This, however, is not a problem since no two 1-morphism are identified in passing from $\FF _Y$ to $\FF _Y / R^\prime$, whence $\CC ^\prime \simeq (\FF _{X ^\prime})_1 = (\FF _Y)_1 \simeq (\FF _Y/R^\prime)_1$.  The other possible obstacle for $\FF _Y/R^\prime$ to qualify as a 2-category is that there may be 2-morphisms in $\FF _Y/R^\prime$ which ought to be identified but which are not, which would mean that the coherence conditions are not satisfied. For example, $*_1$ should be strictly associative.  However, the choice of $R^\prime$ and the fact that $\FF _Y \lra \FF _Y/ R^\prime$ is surjective preclude this from happening.  Therefore, $\FF _Y /R^\prime$ is a 2-category which comes with a map $\FF _X \lra \FF _Y / R^\prime$ in $p_2Cat$.  \\

It only remains to show that $\ca _{X,C}$ has the desired universal property.  Given $D\in 2Cat$ and $F : X \lra D$ in $gr_2Cat$ such that $F : \FF _X \lra D$ identifies objects related by $C$, then $F$ induces a map $F : \FF _{X^\prime} \lra D$ by Theorem \ref{thm1}.  The map $X^\prime \lra D$ in $gr_2Cat $ extends uniquely to a map $Y \lra D$ because there is only one possible choice of where to send each $\al _{f,g,h}$, $\rho _f$, $\lambda _f$, namely the structure maps $\al _{F(f), F(g), F(h)}$, $\lambda _{F(f)}$, $\rho _{F(f)}$ in $D$.  The map from $\FF _Y$ already has the property that $Fx = Fy$ if $xCy$ (Here we abuse notation and denote all maps by $F$). The only additional relations in $C^\prime$ are the coherences conditions for 2-categories.  These relations will automatically become equalities in $D$ because $D$ is a 2-category.  Thus, $xR^\prime y$ implies $Fx = Fy$, whence $F: \FF _Y \lra D$ descends to $\FF _Y / R^\prime  \lra D$ uniquely. This map $\FF _Y / R^\prime \lra D$ preserves the maps $\al $, $\la$, $\rho$, so it is a map of 2-categories.  
\end{proof}

\begin{rem} If the original conditions $C$ are such that no two 1-morphism in $\CC _X$ are identified in $\FF _X$ by the equivalence relation $\sim$ of Lemma \ref{firstquotientlem}, then we may initially include the 2-category data $\al _{f,g,h}$, $\lambda _f$, $\rho_f$ and conditions in the original data and conditions and find that $\FF _{X/C}$ is already a 2-category.  The only obstacle to doing this in general is that there may be morphism in $\CC _X ^\prime$ which were not in $\CC _X$.  
\end{rem}

Theorem \ref{thm2} has the unusual property that it makes reference to pre-2-categories in the description of $\ca _{X,C}$.  The following corollary justifies calling $\ca _{X}$ the 2-category generated by $X$.
  
\begin{cor} Consider any generating data $X= (X_2 \rightrightarrows \CC _{X_1})$.  
\ben $\ca _X$ has the following universal property.  There is a canonical inclusion $\io _A : X \lra \ca _X$ of 2-categorically graded sets, and for any 2-category $\BB$ with an inclusion $\io _B : X \lra \BB$ such that $(\io _B) _{|\CC_{X_1}}$ is a map in $p_2Cat$, there is a unique map of 2-categories $F : \ca _X \lra \BB$ such that $\io _B = F \io _A$.   
\i If $X$ is generating data and $C$ is a binary relation on $\ca _X$, then there exists a 2-category $\ca _{X/ C}$ satisfying:
\ben There is a map of 2-categories $G : \ca _X \lra \ca_{X/C}$ such that $xCy$ implies $Gx = Gy$, and
\i Any other map $F : \ca _X \lra \BB$ of 2-categories such that $xCy$ implies $Fx = Fy$ factors uniquely through $G$.
\een
\i Any 2-category has a presentation in terms of generators and relations, i.e. any $\BB \in 2Cat$ is isomorphic to some $\ca _{X/C}$ for some generating data $X$ and binary relation $C$ on $\ca _X$.
\een
\end{cor}

\begin{proof}
\ben By Lemma \ref{lem1}, to have such a map $\io _B$ is the same as having a map $\FF _X \lra \BB$ in $p_2Cat$. The result now follows directly from Theorem \ref{thm2}.  
\i Let $Y = X\sqcup \{\al _{f,g,h} , \lambda _f , \rho _f \}_{f,g,h \in (\ca _X)_1}$.  We have $\FF _X \hookrightarrow \FF _Y \stackrel{\pi}\lra \ca _X \lra \ca _{Y , \pi \inv C}$.  Since $F: \ca _X \lra \BB$ identifies objects related by $C$, $F \pi$ identifies objects related by $\pi \inv C$, whence $F\pi$ factors uniquely through $\ca _{X/C}:= \ca _{Y, \pi \inv C}$ via some map $H : \ca _{X/C} \lra \BB$ of pre-2-categories.  Since $\pi $ is surjective, the composition $\ca _X \lra \ca _{X/C} \stackrel{H}\lra \BB$ is $F$.  Note that $\pi \inv C$ contains the coherence conditions for a 2-category, so $\ca _{X/C}$ is a 2-category.   
\i Suppose $\BB$ is a 2-category.  Let $X = \BB \in gr_2Cat$, so $\ca _X \stackrel{p}\lra \BB$ is a surjection.  Let $C$ be the binary relation on $\ca _X$ which relates every two points in the same fiber of $p$.  Then $\ca _{X/C} \simeq \BB$.
\een
\end{proof}

Theorems~\ref{thm1}, \ref{thm2} can be extended to strict 2-categories.  There is more than one approach to extending these results.  This can be done by modifying the proofs to get a strict 2-category given by generators and relations. At the first stage, the construction of the free pre-1-category $\CC _X$ is replaced by the free 1-category, i.e. the path category generated by $X$.  The free strict 2-category $\FF _X$ can be constructed in a similary way.  Alternately, we can view a strict 2-category as a pre-2-category with extra conditions. We can observe that any 2-category is equivalent to a strict 2-category and get a weaker version of \ref{thm2}, or follow the approach in \cite{str} to prove the existence of a free $\om$-category on a set.

\section{The Quotient of a Category by the Action of a Monoidal Category}\label{quotientsection}

\noi For an action of a symmetric monoidal category $\MM$, on a category $\CC$, we define the notion of a quotient $Q = \CC /\MM$, which is a 2-category, and show that such a quotient always exists and is unique up to isomorphism.  

\begin{Def} A monoidal category $(\MM,\otimes,\beta,l,r)$
consists of 
a category $\MM$, a functor
$\otimes:\MM \times \MM \lra \MM$,
an object $1\in Ob(\MM)$
and three isomorphisms of functors
$$
\beta_{a,b,c}:\ (a\otimes b)\otimes c\to a\otimes (b\otimes c),\ \ \
1\otimes a \stackrel{l_a}{\lra}a \stackrel{r_a}{\longleftarrow} a\otimes 1
;$$
that satisfy
\begin{itemize}\item
(AA) consistency (i.e., self-compatibility)
of associativity
called {\em pentagram identity}
$$
\CD
((ab)c)d
@>\beta_{ab,c,d}>>
(ab)(cd)
@>\beta_{a,b,cd}>>
a(b(cd))
\\
@V{=}VV
@.
@A{1_a\otimes \beta_{b,c,d}}AA
\\
((ab)c)d
@>\beta_{a,b,c}\otimes 1_d>>
(a(bc))d
@>\beta_{a,bc,d}>>
a((bc)d)
\endCD
$$
\i
(AU) compatibility
of associativity and unital constraints:
$$
\CD
a\otimes 1\otimes b
@>\beta_{a,1,b}>>
a\otimes 1\otimes b
\\
@V{r_a\otimes 1_b}VV
@V{1_a\otimes l_b}VV
\\
a\otimes b
@>=>>
a\otimes b
\endCD
$$
\end{itemize}
\end{Def}

\begin{Def} Let $\MM$ be a small symmetric monoidal category and $\CC$ a small 1-category.  

\begin{enumerate} 

\item An action of a  monoidal category $(\MM,\otimes,\beta,l,r)$
on a category $\CC$ consists of 
a functor
$A :\MM\times\CC \lra \CC$ (also denoted $ A :(m,a)\mapsto m.a)$,
an object 
and two isomorphisms of functors
$$
\beta^\ast_{m,n,a}:\ (m\otimes n). a\to m. (n. a),\ \ \
1. a \stackrel{u_m}{\lra} a
;$$
that satisfy
\begin{itemize}
\item
(AA) compatibility)
of two associativity constraints
(again a pentagram identity),
$$
\CD
((lm)n)a
@>\beta^\ast_{lm,n,a}>>
(lm)(na)
@>\beta^\ast_{l,m,na}>>
l(m(na))
\\
@V{=}VV
@.
@A{1_l\otimes \beta^\ast_{m,n,a}}AA
\\
((lm)n)a
@>\beta_{l,m,n}\otimes 1_a>>
(l(mn))a
@>\beta^\ast_{l,mn,a}>>
l((mn)a)
\endCD
$$
\item
(AU) compatibility
of associativity and unital constraints:
$$
\CD
(m\otimes 1). a
@>\beta^\ast_{m,1,a}>>
m\otimes (1. a)
\\
@V{r_m. 1_a}VV
@V{1_a. u_a}VV
\\
m. a
@>=>>
m. a
\endCD
$$
\end{itemize}
\een
\end{Def}

We are now ready to define the quotient of a category $\CC$ by an action of a monoidal category $\MM$, but first we recall from \cite{lei} the defintition of \emph{natural transformation} of functors between 2-categories.  Suppose that $\ca$ is 1-category and $\BB$ is a 2-category.  A natural transformation $ F \implies G$ between two functors $F,G : \CC \lra \DD$ of 2-categories consists of a 1-morphism  $\zeta _x : Fx \lra Gx$ for each object $x \in \ca _0$ and a 2-morphism $\eta _f :\zeta_y *_0 Ff \implies Gf *_0 \zeta _x$ for each 1-morphism $x \stackrel {f}\lra y $ in $\ca$ subject to the following conditions. For all $x \in \ca _0$, 

\begin{equation}\label{funceqn2}
\eta _{x} =  \rho \inv _{\zeta _x} *_1 \lambda _{\zeta _x},
\end{equation}
and $\eta$ is functorial in $\ca$.  This means that for all $x \stackrel {f} \lra y \stackrel {g}\lra z$ in $\ca$, 
\begin{equation}\label{funceqn}
\eta_{gf} = \al \inv _{Gg, Gf, \zeta _x}*_1(Gg *_0 \eta _f)*_1 \al _{Gg , \zeta _y , Ff}*_1(\eta _g *_0 Ff)*_1 \al \inv _{\zeta _z , Fg, Ff}.
\end{equation}


Loosely, this says that the diagram 
\begin{diagram}
Fx               & \rTo^{Ff}   & Fy  & \rTo^{Fg}  & Fz\\
\dTo ^{\zeta _x} & \ldImplies^{\eta _f} & \dTo^{\zeta _y} & \ldImplies ^{\eta _g} & \dTo^{\zeta _z}\\
Gx               & \rTo^{Gf}   & Gy  & \rTo^{Gg}  & Gz
\end{diagram}
coincides with 
\begin{diagram}
Fx               & \rTo^{F(gf)}   & Fz\\
\dTo ^{\zeta _x} & \ldImplies^{\eta _{gf}} & \dTo^{\zeta _z}\\
Gx               & \rTo^{G(gf)}   &  Gz.
\end{diagram}
These diagrams give the rough idea, but since composition in $\BB$ is not strictly associative, the diagrams are ambiguous. The precise statement is given above in equation \eqref{funceqn}.

\begin{Def}\label{quodef}
 A quotient $\CC /\MM $ of an action of $\MM$ on $\CC$  consists of a tiple $(Q, \pi , \theta)$, where $Q$ is a 2-category, $\pi : \CC \lra Q$,  and $\theta : \pi \circ p _2 \Longrightarrow \pi \circ A $ is a natural transformation in 2-Cat, where $\pi p_2$ and $ \pi A : \MM \times \CC \lra Q$.  We ask that for any other such $(Q^\prime , \pi ^\prime , \theta ^\prime )$, $\pi ^\prime $ factors uniquely through $\pi$ via some map $F$ such that $F \theta = \theta ^\prime$. 
\end{Def}

We now offer an explicit description of a quotient $(Q, \pi , \theta )$.  Letting $\theta = (\eta , \zeta)$, the quotient $(Q,\pi , \eta)$ is given by Q1-Q7 listed below.  Since $\theta$ is a morphism with  source $\MM \times \CC$, a sufficient condition for functoriality of $\theta$ is that $\eta$ is functorial in $\CC$ and $\MM$ independently (Q3, Q4) and that the $\eta ^f _a$'s are compatible with the $\eta ^m _x$'s (Q6).  To see this, observe that any 1-morphism $(f,x) \in \MM \times \CC$ can be decomposed as $(f,1)*_0 (1,x)$ or $(1,x)*_0(1,f)$.  Hence, $\theta$ is determined by its values on morphisms of the form $(f,1)$ and $(1,x)$.  To be functorial, $\theta$ must be functorial in each direction and take the same value on both possible decompositions of $(f,x)$. \\

A quotient $(Q, \pi , \theta)$ of $\CC $ by $\MM$ is equivalent to the following data and conditions.
\begin{itemize}
\item (Q1) a 2-category $Q$ together with a functor $\pi : \CC \lra Q$.

\item (Q2) 1-morphisms $\zeta ^m _a : \pi (a) \lra \pi (m.a)$ in $Q$ for each $a \in \CC _0$, $m \in \MM _0$.  

\item (Q3) 2-morphisms $\eta ^m _x : \pi (x \otimes m) *_0 \zeta ^m _a \Longrightarrow \zeta ^m _b *_0  \pi x$ for each $x \in Hom_\CC (a,b)$, $m \in \MM _0$ such that $\eta$ is functorial in $\CC$.  In other words, $\eta ^m _x$ fits into a diagram
\begin{diagram}
\pi (a)            & \rTo^{\pi (x)}           & \pi (b)\\
\dTo^{\zeta ^m _a} & \ldImplies^{\eta ^m _x}   & \dTo_{\zeta ^m _b}\\
\pi (m.a )         & \rTo_{\pi (x \otimes m)} & \pi (m.b).
\end{diagram}



For $\eta$ to be functorial in $\CC$ means simply that given $a \stackrel{x}{\lra} b \stackrel{y}{\lra} c$ in $\CC$,
 \begin{diagram}
\pi (a)            & \rTo^{\pi (x)}           & \pi (b)            & \rTo^{\pi(y)}           & \pi (c) \\
\dTo^{\zeta ^m _a} & \ldImplies^{\eta ^m _x}  & \dTo_{\zeta ^m _b} & \ldImplies^{\eta ^m _y} & \dTo^{\zeta ^m _c}\\
\pi (m.a )         & \rTo_{\pi (x \otimes m)} & \pi (m.b)          & \rTo_{\pi(y \otimes m)} & \pi (m.c) 
\end{diagram}
coincides with
\begin{diagram}
\pi (a)            & \rTo^{\pi (yx)}           & \pi (c)\\
\dTo^{\zeta ^m _a} & \ldImplies^{\eta ^m _{yx}}   & \dTo_{\zeta ^m _c}\\
\pi (m.a )         & \rTo_{\pi (yx \otimes m)} & \pi (m.c)
\end{diagram}
in the sense of equation \eqref{funceqn}.


\item (Q4) 2-morphisms $\eta ^f _a :  \pi (a \otimes f)  *_0 \zeta ^m _a \Longrightarrow \zeta ^m _a$ for each $f \in Hom_\MM (m,n)$, $a \in Ob(\CC)$ such that $\eta$ is functorial in $\MM$.  In other words, $\eta ^f _a$ fits into a diagram 
\begin{diagram} 
\pi (a ) & \rTo^{=} & \pi (a)\\
\dTo^{\zeta ^m _a} & \ldImplies^{\eta ^f _a} & \dTo_{\zeta ^n _a}\\
\pi (m.a) & \rTo_{\pi (a \otimes f)} & \pi (n.a) 
\end{diagram}
such that for all $l\stackrel{f}{\lra} m \stackrel {g}{\lra} n$ in $\MM$, the diagram
\begin{diagram} 
\pi (a )           & \rTo^{=}                & \pi (a)             & \rTo^{=}                & \pi(a) \\
\dTo^{\zeta ^l _a} & \ldImplies^{\eta ^f _a} & \dTo_{\zeta ^m _a}  & \ldImplies^{\eta ^g _a} & \dTo_{\zeta ^n _a} \\
\pi (l.a) & \rTo_{\pi (a \otimes f)} & \pi (m.a)  & \rTo_{\pi (a \otimes g)} & \pi (n.a)
\end{diagram}
coincides with 
\begin{diagram} 
\pi (a ) & \rTo^{=} & \pi (a)\\
\dTo^{\zeta ^l _a} & \ldImplies^{\eta ^{gf} _a} & \dTo_{\zeta ^n _a}\\
\pi (l.a) & \rTo_{\pi (a \otimes gf)} & \pi (n.a) 
\end{diagram}
in the sense if equation \eqref{funceqn}.

\i (Q5) For $a \in \CC _0$, $m \in \MM _0$, $\eta ^m _a$ of Q3 and Q4 are the same, and equation \eqref{funceqn2} is satisfied.

\i (Q6) The $\eta$'s are compatible in the sense that the following two diagrams of 2-morphisms are ``identical'' in the sense of equation \eqref{funceqn}.

\begin{diagram}
\pi (a )           & \rTo^{=}                 & \pi (a)            & \rTo^{\pi (x)}          & \pi(b)  \\
\dTo^{\zeta ^m _a} & \ldImplies^{\eta ^f _a}  & \dTo_{\zeta ^n _a} & \ldImplies^{\eta ^n _x} & \dTo^{\zeta ^n _b}\\
\pi (m.a)          & \rTo_{\pi (a \otimes f)} & \pi (n.a)          & \rTo_{\pi (x\otimes n)} & \pi (n.b)
\end{diagram}

\begin{diagram}
\pi (a )           & \rTo^{\pi (x)}                 & \pi (b)            & \rTo^{=}          & \pi(b)  \\
\dTo^{\zeta ^m _a} & \ldImplies^{\eta ^m _x}  & \dTo_{\zeta ^m _b} & \ldImplies^{\eta ^f _b} & \dTo^{\zeta ^n _b}\\
\pi (m.a)          & \rTo_{\pi (x \otimes m)} & \pi (m.b)          & \rTo_{\pi (b\otimes f)} & \pi (n.b)
\end{diagram}

\item (Q7) $Q$ is universal with respect to these properties, i.e. for any other 2-category $(\pi ^\prime : \CC \lra Q^\prime, \zeta ^\prime, \eta ^\prime)$ satisfying (Q1)-(Q4), $\pi ^\prime $ factors uniquely through $\pi : \CC \lra Q$.
\ei

As a corollary to Theorem~\ref{thm2}, the existence of a quotient is guaranteed.  

\begin{prop}\label{quoprop} Given a category $\CC$ with an action of a symmetric monoidal category $\MM$, there exists a quotient 2-category $\CC /\MM$, which is unique up to isomorphism.
\end{prop}
\begin{proof} We let $X $ be the union of the following data:
\begin{enumerate} 
\item  $\CC$
\item  a 1-morphism $\zeta ^m _a : a \lra m.a$ for each $m\in \MM _0$, $a \in \CC _0$
\item  a 2-morphism $\eta ^m _x :  \zeta ^m _b *_0 x \Longrightarrow x\otimes a *_0 \zeta ^m _a$ for each $ a \stackrel{x}\lra b$ in $\CC$ and $m \in \MM _0$
\item  a 2-morphism $\eta ^f _a :  \pi (a \otimes f)  *_0 \zeta ^m _a \Longrightarrow \zeta ^m _a$ for each $f \in Hom_\MM (m,n)$, $a \in \CC _0$
\end{enumerate}



\noi More concretely, let $X _1 = \CC \cup \{\zeta ^m _a \} _{(m,a)\in \MM_0 \times \CC _0}$, let $X_2 $ be the set of all $\eta$'s, and $X = X_1 \cup X_2$.  This generating data produces a free pre-2-category $\FF _X$.  We let $C$ be the conditions described in Q3-Q6 together with the relations needed to make the pre-1-category generated by $Ob(\CC) \subset \FF _X$ into a strict 1-category isomorphic to $\CC$. That is to say, we include the following relations. Let $\circ$ denote composition in $\FF _X$, and $*_0$ denote composition in $\CC$. For each $f$,$g \in \CC$, the relation $f\circ _0g = f*_0g$ is in $C$. Also, $C$ contains the relations $(f\circ_0g)\circ_0 h = f\circ_0(g\circ_0 h)$ for each $f,g,h \in \CC$ for which composition is defined.  The final relations needed are $f\circ _0 s_0 f = f = t_0f \circ _0 f$ as well as $\al _{f,g,h} = (h*_0g)*_0f$, $\la _f = f$, and $\rho _f = f$.\\

\noi With these relations $C$, we attain the 2-category $Q = \ca _{X,C}$. The conditions in $C$ which relate morphisms in $Ob(\CC) \subset \FF _X$ are chosen precisely so that $Ob(\CC) \hookrightarrow \FF _X \lra \ca _{X,C}$ induces a morphism of 2-categories $\pi : \CC \lra \ca _{X,C}$.  Since $\FF _X$ maps to $\ca _{X,C} $, $\ca _{X,C}$ clearly has the 1-morphisms, $\zeta ^m _a$ and 2-morphisms $\eta ^f _a$, $\eta ^m _x$ needed to be a quotient category.  The conditions $C$ were chosen exactly so that the relations described in Q3-Q6 hold in $\ca _{X,C}$.  The universal property of $\ca _{X,C}$ as the 2-category generated by $X$ with relations $C$ implies that the universal property Q7 holds for $\ca _{X,C}$.  The uniqueness of $\ca _{X,C}$ is a consequence of the universal property Q7.  
\end{proof}

\subsection{Variations}

Definition \ref{quodef} gives the quotient as a sort of asymmetrical colimit.  However, the proof of Proposition \ref{quoprop} can be modified slightly to accomodate variations of Definition \ref{quodef}.  For instance, one can attain a more symmetric version of $Q$ with maps $a \lra m.a$ and maps $m.a \lra a$.  This can be accomplished by asking for another natural transformation $\phi : \pi A \implies \pi p_2$ and modifications $id_{\pi p_2}   \Rrightarrow \phi \theta$ and $id_{\pi A} \Rrightarrow \theta \phi$ with inverses.  Alternatively, we could request that $\theta$ and $\phi$ are inverses of each other and get a stricter version.  In another variation of Definition \ref{quodef}, we may also want to include in $Q$ 2-morphisms $\varphi ^{m,n} _a : \zeta ^{mn} _a \Longrightarrow \beta ^* *_0 \zeta ^m _{na} *_0 \zeta ^n _a$ and $\xi ^{l,m,n} _a : \beta ^* *_0 \zeta ^{(lm)n} _a \Longrightarrow \zeta ^{l(mn)} _a$ satisfying a large coherence diagram. This has the effect of demanding that the choice of $\zeta$ is compatible with the tensor product in $\MM$. \\



\end{document}